\nonstopmode

\documentclass[11pt]{article}

\usepackage{geometry}
\usepackage{amsmath,amssymb,amsthm}
\usepackage{graphics,epsfig,calc}
\usepackage{amsmath}
\usepackage{latexsym,epsfig,bm,amssymb}
\usepackage{color}
\usepackage{amsthm,mathrsfs}

\usepackage{mathptmx}

\DeclareSymbolFont{AMSb}{U}{msb}{m}{n}
\DeclareSymbolFontAlphabet{\mathbb}{AMSb}

\newcommand{\beqn}{\begin{eqnarray}}
\newcommand{\eeqn}{\end{eqnarray}}
\newcommand{\be}{\begin{equation}}
\newcommand{\ee}{\end{equation}}
\newcommand{\ba}{\begin{array}}
\newcommand{\ea}{\end{array}}
\newcommand{\bo}{{\hfill\loota}}
\newcommand{\loota}{\hbox{\enspace{\vrule height 7pt depth 0pt width 7pt}}}

\newcommand{\cC}{{\cal C}}
\newcommand{\cD}{{\cal D}}

\newcommand{\cF}{{\cal F}}
\newcommand{\cH}{{\cal H}}

\newcommand{\cM}{{\cal M}}
\newcommand{\cO}{{\cal O}}

\newcommand{\cR}{{\cal R}}
\newcommand{\cS}{{\cal S}}

\newcommand{\cV}{{\cal V}}
\newcommand{\cW}{{\cal W}}

\newcommand{\cX}{{\cal X}}

\newcommand{\al}{\alpha}

\newcommand{\ci}{\cite}
\newcommand{\de}{\delta}
\newcommand{\De}{\Delta}
\newcommand{\ds}{\displaystyle}
\newcommand{\fr}{\frac}
\newcommand{\ga}{\gamma}

\newcommand{\la}{\label}

\newcommand{\Lam}{\Lambda}
\newcommand{\na}{\nabla}
\newcommand{\Om}{\Omega}
\newcommand{\om}{\omega}

\newcommand{\ov}{\overline}

\newcommand{\pa}{\partial}
\newcommand{\re}{\ref}

\newcommand{\Si}{\Sigma}
\newcommand{\si}{\sigma}
\newcommand{\ti}{\tilde}
\newcommand{\dist}{\rm dist\5}
\newcommand{\ve}{\varepsilon}

\newcommand\C{{\mathbb C}}
\newcommand\R{{\mathbb R}}

\newcommand\T{{\mathbb T}}
\newcommand\Z{{\mathbb Z}}

\newcommand{\Ga}{\Gamma}

\newcommand{\vka}{\varkappa}

\newcommand{\cm}{{\rm m}}
\newcommand\nab{{\bf \nabla}}

\newcommand{\5}{{\hspace{0.5mm}}}

\newcommand{\3}{{\hspace{0.2mm}}}

\newcommand{\const}{\mathop{\rm const}\nolimits}

\newcommand{\supp}{\mathop{\rm supp}\nolimits}

\newcommand{\spec}{\mathop{\rm Spec\,}\nolimits}
\newcommand{\rRe}{{\rm Re\5}}
\newcommand{\rIm}{{\rm Im\5}}

\renewcommand{\Pr}{\hspace{-6mm}{\bf Proof~}}
\newcommand{\mes}{{\rm mes~}}

\newcommand{\Ran}{{\rm Ran\3}}

\newcommand{\toYs}{\stackrel{\cX_\al}{-\!\!-\!\!\!\longrightarrow}}

\renewcommand{\theequation}{\thesection.\arabic{equation}}
\newtheorem{theorem}{Theorem}[section]
\renewcommand{\thetheorem}{\arabic{section}.\arabic{theorem}}
\newtheorem{definition}[theorem]{Definition}

\newtheorem{lemma}[theorem]{Lemma}
\newtheorem{example}[theorem]{Example}
\newtheorem{remark}[theorem]{Remark}
\newtheorem{remarks}[theorem]{Remarks}
\newtheorem{cor}[theorem]{Corollary}
\newtheorem{proposition}[theorem]{Proposition}

\newcommand{\bd}{\begin{definition}}
\newcommand{\ed}{\end{definition}}
\newcommand{\bt}{\begin{theorem}}
\newcommand{\et}{\end{theorem}}
\newcommand{\bqt}{\begin{qtheorem}}
\newcommand{\eqt}{\end{qtheorem}}

\newcommand{\bp}{\begin{proposition}}
\newcommand{\ep}{\end{proposition}}

\newcommand{\bl}{\begin{lemma}}
\newcommand{\el}{\end{lemma}}
\newcommand{\bc}{\begin{cor}}
\newcommand{\ec}{\end{cor}}

\newcommand{\bex}{\begin{example}}
\newcommand{\eex}{\end{example}}
\newcommand{\bexs}{\begin{examples}}
\newcommand{\eexs}{\end{examples}}

\newcommand{\bexe}{\begin{exercice}}
\newcommand{\eexe}{\end{exercice}}

\newcommand{\br}{\begin{remark} }
\newcommand{\er}{\end{remark}}
\newcommand{\brs}{\begin{remarks}}
\newcommand{\ers}{\end{remarks}}

\begin{document}

\begin{titlepage}

\begin{center}
{\Large\bf On the dispersion decay  for crystals in 
\medskip\\
the linearized Schr\"odinger--Poisson model
}
\end{center}
\bigskip\bigskip

\begin{center}
{\large A. Komech}
\footnote{
The research was carried out at the IITP RAS at the expense of the Russian Foundation for Sciences (project 14-50-00150).}
\\
{\it Faculty of Mathematics of Vienna University\\
and Institute for Information Transmission Problems RAS } \\
e-mail:~alexander.komech@univie.ac.at
\bigskip\\
{\large E. Kopylova}
\footnote{
The research was carried out at the IITP RAS at the expense of the Russian Foundation for Sciences (project 14-50-00150).}
\\
{\it Faculty of Mathematics of Vienna University\\
and Institute for Information Transmission Problems RAS} \\
e-mail:~elena.kopylova@univie.ac.at
\end{center}
\vspace{1cm}

\begin{abstract}
The 
Schr\"odinger--Poisson--Newton  equations for crystals
with a~cubic lattice and one ion per cell are considered.
The ion charge density  is assumed  i) to satisfy the Wiener and Jellium conditions
 introduced in our previous paper \ci{KKpl2015}, and
 ii) to be exponentially decaying at infinity.
The corresponding examples are given.

 We study the linearized dynamics at the ground state.
The dispersion relations are introduced via
spectral resolution for the non-selfadjoint Hamilton generator 
using the positivity of the energy  established in \ci{KKpl2015}.

Our main result is 
 the dispersion decay
in the weighted Sobolev norms
for solutions with initial states from the space of  continuous spectrum of the Hamilton generator.
We also prove
the absence of singular spectrum and 
limiting absorption principle.
The multiplicity of every eigenvalue is shown to be infinite.

The proofs rely on novel exact bounds and compactness
for the inversion of the Bloch generators and on 
uniform asymptotics for the dispersion relations.
We derive the bounds 
by the energy positivity from \ci{KKpl2015}.
We also use the  theory of analytic sets.

\end{abstract}



\end{titlepage}

\section{Introduction}
First mathematical results on the stability of matter were obtained
by Dyson and Lenard
in \ci{D1967, DL1968}, where the energy bound from below
was established.
The thermodynamic limit
for the Coulomb systems
was first studied by Lebowitz and Lieb
\ci{LL1969,LL1973}, see the survey and further development in \ci{LS2010}.
These results were extended  by Catto, Le Bris, Lions,
and others to the Thomas--Fermi and Hartree--Fock models
\ci{CBL1998,CBL2001,CBL2002}.
Further results in this direction are due to  Canc\'es,  Lahbabi, Lewin, Sabin, Stoltz, and others
 \ci{CLL2013,CS2012, BL2005, LS2014-1, LS2014-2}.
All these results concern  either
the convergence of
the ground state of finite particle systems
in the thermodynamic limit or
the existence of the ground state
for infinite particle systems.
Canc\`es and Le Bris proved the well-posedness
for the Hartree--Fock molecular system \ci{CB1999}.

However, no attention was paid to
the dynamical stability of crystals
with moving ions.
This stability is
necessary for a rigorous analysis
of fundamental quantum phenomena in the solid state physics:
heat conductivity, electric conductivity, thermoelectronic emission, photoelectric effect,
Compton effect,
etc., see \ci{BLR}.
\medskip

We consider
the simplest Schr\"o\-din\-ger--Poisson model for the crystal
with one ion per cell.
The electron cloud is described by  the one-particle Schr\"odinger equation;
 the ions are looked upon as  particles
that correspond to the Born and Oppenheimer  approximation.
The ions interact with the electron cloud via
the scalar potential, which  is a~solution to the corresponding Poisson equation.

This model does not respect the Pauli exclusion principle for electrons.
Nevertheless, it provides a convenient framework
to introduce suitable functional tools that might be instrumental 
for physically more realistic models (the Thomas--Fermi, Hartree--Fock, and second quantized models).
\medskip

The  Schr\"o\-din\-ger--Poisson equations are extensively studied for 
molecular systems
since P.-L. Lions paper
\ci{Lions1981}, see \ci{A2008, Catto2013, Nier1993} 
and the references therein. 
In present paper, we establish
the dispersion decay  for the linearized Schr\"o\-din\-ger--Poisson equations at
the ground state for infinite crystals.
The decay is proved
under the Jellium and Wiener conditions on the ion charge density introduced in \ci{KKpl2015}.
The ground state for this model was constructed in \ci{K2014}.
\medskip

We denote by
$\sigma(x)$ the charge density of one ion, and define $Z>0$ by the identity 
\be\la{ro+}
\int_{\R^3} \sigma(x)\,dx=eZ>0,
\ee
where $e>0$ is the elementary charge.
We assume througout the paper that
\be\la{L123i}
(\De-1)\si\in L^1(\R^3)
\ee
which provides a suitable decay for the Fourier transform of $\si$.
In particular, the  series (\re{W1}) are converging.
Moreover, we  assume
 the exponential decay of the ion charge density
\be\la{rd}
 |\si(x)|\le Ce^{-\ve|x|},~~~~x\in\R^3,
\ee
where $\ve>0$.
The cubic
lattice  $\Ga= \Z^3$ is chosen for the simplicity of notations.
Let $\psi(x,t)$ be the wave function of the electron field,
$q(n,t)$ denote the ions displacements,
and
$\Phi(x,t)$ be the electrostatic  potential generated by the ions and electrons.
We assume that $\hbar=c=\cm=1$, where $c$ is the speed of light and $\cm$ is the electron mass.
The coupled Schr\"odinger--Poisson--Newton equations  read as
\beqn\la{LPS1}
i\dot\psi(x,t)\!\!&=&\!\!-\fr12\De\psi(x,t)-e\Phi(x,t)\psi(x,t),\qquad x\in\R^3,
\\
\nonumber\\
-\De\Phi(x,t)\!\!&=&\!\!\rho(x,t):=\sum_{n\in\Z^3}\sigma(x-n-q(n,t))-e|\psi(x,t)|^2,\qquad x\in\R^3,
\la{LPS2}
\\
\nonumber\\
M\ddot q(n,t)
\!\!&=&\!\!-\langle\nab\Phi(x,t),\sigma(x-n-q(n,t))\rangle,
\qquad n\in\Z^3.
\la{LPS3}
\eeqn
Here $t\in\R$, the
brackets
 stand for the Hermitian scalar product on the real  Hilbert
space $L^2(\R^3)$ and for its various extensions, the  series (\re{LPS2}) converges in
a suitable sense, and $M>0$ is the mass of one ion.
All the derivatives here and below are understood in the sense of distributions.
The potential $\Phi(x,t)$ can be eliminated using the operator $G:=(-\De)^{-1}$ defined by
\be\la{GG}
G\rho(x):=\ds\fr1{4\pi}\int\fr{\rho(y)dy}{|x-y|},\qquad x\in\R^3.
\ee
Then equations (\re{LPS1})--(\re{LPS3}) can be  formally written as a~Hamilton system
with the Hamilton  functional (energy)
\be\la{Hfor}
 E(\psi,q,p)=\fr12\int_{\R^3}[|\na\psi(x)|^2+\rho(x)G\rho(x)]dx+\sum_{n} \fr{p^2(n)}{2M},
\ee
where
$ q:=(q(n): ~n\in\Z^3)$, $p:=(p(n): ~n\in\Z^3)$,  $\rho(x)$ is defined
similarly to (\re{LPS2}).
Namely, system (\re{LPS1})--(\re{LPS3}) can
be formally written as
\be\la{HSi}
i\dot \psi(x,t)=\pa_{\ov \psi(x)}\cH,
~~~
\dot q(n,t)=\pa_{p(n)}\cH,
~~~
\dot p(n,t)=-\!\pa_{q(n)}\cH,
\ee
where $\pa_{\ov z}:=\fr12(\pa_{z_1}+i\pa_{z_2})$ with $z_1=\rRe z$ and $z_2=\rIm z$.
\medskip

A ground state of a crystal is a $\Ga$-periodic  solution
\be\la{gri}
\psi^0(x)e^{-i\om^0 t}~,~~~ \Phi^0(x)~,~~~~q^0(n)=q^0~~{\rm and}~~p^0(n)=0~~~~
{\rm for}~~ n\in\Z^3
\ee
with a real $\om^0$ and minimal energy per cell.
Such ground state
was constructed in \ci{K2014} for general lattice with several ions per cell.
In our case
the ion position  $q^0\in\R^3$ can be chosen arbitrarily, and we set
$q^0=0$ everywhere below.
In our  framework $\psi^0(x)$ will be  a real function up to a phase factor, see (\re{gri}) below.
This factor can be  neglected due to $U(1)$-symmetry of equations
(\re{LPS1})--(\re{LPS3}),
and respectively, we will consider the real ground states $\psi^0(x)$.
\medskip

In present paper, we prove
the dispersion decay for  the
{\it formal linearization}
of the nonlinear  system (\re{LPS1})--(\re{LPS3})
at the ground state (\re{gri}).
The linearization is obtained on substituting
$
 \psi(x,t)=[\psi^0(x)+\Psi(x,t)]e^{-i\om^0 t}
$
into the nonlinear equations (\re{LPS1}), (\re{LPS3}) with $\Phi(x,t)=G\rho(x,t)$ and retaining the
linear terms in $Y(t)=(\Psi(\cdot,t),q(\cdot,t),p(\cdot,t))$.
For the real ground state
the linearized equation reads as follows (see \ci[(1.14)]{KKpl2015}):
\be\la{JDi}
\dot Y(t)=AY(t),\qquad
 A=\left(\ba{ccrl}
 0   &  H^0 &  0  & 0\medskip\\
-H^0-2e^2\psi^0G\psi^0 &   0  & -S  & 0\\
        0              &   0  &  0  & M^{-1}\\
     -2S^{\5*}         &   0  & -T  &  0\\
\ea\right).
\ee
Here, we denote
$Y(t)=(\Psi_1(\cdot,t),\Psi_2(\cdot,t),q(\cdot,t),p(\cdot,t))$,
where
$\Psi_1(x,t):=\rRe\Psi(x,t),\Psi_2(x,t):=\rIm \Psi(x,t)$,
$H^0:=-\fr12\De-e\Phi^0(x)-\om^0$,
the operators $S$ and $T$
are defined in Appendix,
and
$\psi^0$
denotes the operators of  multiplication
by the real function $\psi^0(x)$.
From the Hamilton representation (\re{HSi}) we have 
\be\la{AJB}
A=JB,\qquad B=
\left(\ba{cccl}
 2H^0+4e^2\psi^0 G\psi^0 & 0 & 2S & 0
 \medskip\\
 0 & 2H^0  &0 & 0\medskip\\
 2S^{\5*}  &    0  &   T    & 0  \\
      0      &    0            &   0    &  M^{-1} \\
\ea\right),\qquad
J=\left(\ba{cccc}
0  & \fr12 &  0 & 0\\
-\fr12 & 0 &  0 & 0\\
0 & 0 &  0 & 1\\
0 & 0 & -1 & 0
\ea\right).
\ee
 The results \ci{KKpl2015} imply that the energy operator $B$ is densely defined 
and is selfadjoint on the Hilbert space
\be\la{cX0}
\cX^0(\R^3):=L^2(\R^3)\oplus L^2(\R^3)\oplus l^2(\Z^3)\oplus l^2(\Z^3).
\ee
Our main goal is  to show the dispersion decay of solutions to (\re{JDi})
in the weighted norms
\be\la{cXw}
\Vert X\Vert_\al:=\Vert \langle x\rangle^\al\,\, \Psi_1(x)\Vert_{L^2(\R^3)}+
\Vert \langle x\rangle^\al \,\, \Psi_2(x)\Vert_{L^2(\R^3)}+
\Vert \langle n\rangle^\al \, q(n)\Vert_{l^2(\Z^3)}+
\Vert \langle n\rangle^\al \, p(n)\Vert_{l^2(\Z^3)}
\ee
with $\al<0$ for $X=(\Psi_1,\Psi_2,q,p)\in\cX^0(\R^3)$.
We will prove the decay
under two following  conditions  C1 and C2  on $\si$ introduced in \ci{KKpl2015}.
\be\la{W1}
{\rm C1.}~\mbox{\bf The Wiener Condition:}~~  \Si(\theta):=\sum_m\Big[
 \fr{\xi\otimes\xi}{|\xi|^2}|\hat\si(\xi)|^2\Big]_{\xi=2\pi m+\theta}>0~,
 \quad\mbox{for a.e.}\,\, \theta\in \Pi^*\setminus\Ga^*.~
\ee
Here $\hat\si(\xi)$ stands for the Fourier transform $\ds\int e^{i\xi x}\si(x)\,dx$,
$\Pi^*:=[0,2\pi]^3$ is the Brillouin zone,
$\Ga^*:=2\pi\Z^3$\, and   $\xi\otimes\xi$ denotes the matrix $\xi_i\xi_j$.
The series of the matrices converges by (\re{L123i}), and the sum is a positive definite matrix.

This condition is
 an analog of the Fermi Golden Rule for crystals. It means a strong coupling 
 of the ions to the electron field.
 
 \be\la{Wai}
{\rm C2.}  ~~~~~~~\mbox{\bf The Jellium Condition:}~~~~~~~~~~~~ \hat\si(2\pi m)=0,\quad m\in\Z^3\setminus 0.
~~~~~~~~~~~~~~~~~~~~~~~~~~~~~~~~~~~~~~~~~~~~~~~~~~~~~~~~~
\ee
This condition  cancels the negative
energy which is provided by the electrostatic instability
 (`Earnshaw's Theorem' \ci{Stratton},
see \ci[Remark 10.2]{KKpl2015}).
It implies that
the periodized ions charge density
corresponding to the ground state
is a positive constant everywhere in the space.
 
 The simplest example of such a 
$\sigma$ is a constant over the unit cell of a given lattice, which is what physicists 
usually call Jellium \cite{GV2005}. Moreover,
this condition  holds 
for a broad class of  functions  $\si$,
see  \ci[Section B.2]{KK2017a}. 
Here we study this model in the rigorous context of the Schr\"odinger--Poisson equations.

Under condition (\re{Wai}),
 the minimum of energy  per  cell
corresponds to the opposite uniform negative
 electronic charge. So
these ion and electronic densities cancel each other,
and the potential $\Phi(x,t)$ vanishes by (\re{LPS2}), see Lemma 2.1 in \ci{KKpl2015}:
\be\la{ppo}
\psi^0(x)\equiv e^{i\phi }\sqrt{Z},\quad\phi \in [0, 2\pi];\quad
\Phi^0(x)\equiv 0,\quad\om^0=0.
\ee
We give examples satisfying all our
conditions  (\re{L123i}),  (\re{rd}),  (\re{W1}), and (\re{Wai}), (see Example \re{ex2}).
\medskip

The key result of \ci{KKpl2015} is the positivity $B>0$ of the energy operator (\re{AJB}) under the Wiener and Jellium conditions C1 and C2. 
Denote by $\cW$  the completion of the space
$\cX^1(\R^3):=H^1(\R^3)\oplus H^1(\R^3)\oplus l^2(\Z^3)\oplus l^2(\Z^3)$ with the norm
\be\la{cW}
\Vert Y\Vert_\cW:= \Vert\Lam Y\Vert_{\cX^0(\R^3)},\qquad \Lam:= B^{1/2}>0.
\ee
In  \ci{KKpl2015} we have proved 
 that for any $Y(0)\in\cW$ there exists a unique  {\it weak solution}  $Y(t)\in C(\R,\cW)$ to (\re{JDi}).
 The main result of the present paper is the following  theorem.

\bt\la{tmg}
 Let  conditions \eqref{L123i},  \eqref{rd},  \eqref{W1}, and \eqref{Wai}  hold.
Then every solution $Y(t)\in C(\R,\cW)$ to (\re{JDi}) splits as follows
\be\la{FY53}
 Y(t)=\sum_1^M Y_k e^{-i\om^*_k t}+Y_c(t),
\ee
where   $M\le\infty$ and  $Y_k\in\cW$. Moreover,
\be\la{Minf}
|\om^*_k|\to\infty,\qquad k\to\infty,
\ee
if $M=\infty$. The remainder $Y_c(t)$ decays in the weighted norms: for any $\al<-3/2$,
\be\la{FY5}
\Vert \Lam Y_c(t)\Vert_\al\to 0,~~~~~~|t|\to\infty.
\ee
\et
This theorem  means the linear  asymptotics stability of the ground state (\re{ppo})
when $M=0$.
\medskip

Let us comment on our approach.
We   develop our methods \ci{KKpl2015}
relying on      the Bloch transform.
Namely, the generator  $A$
commutes with  translations by vectors from $\Ga$.
Hence, the  equation
\eqref{JDi} can be reduced using the
Fourier--Bloch--\allowbreak Gelfand--\allowbreak Zak
transform
$Y(t)\mapsto\ti Y(\cdot,t)\in L^2(\Pi^*, \cX^0(\T))$, 
where
$\T:=\R^3/\Ga$ is the periodic cell, and
\be\la{cX03}
\cX^s(\T):= H^s(\T)\oplus H^s(\T)\oplus \C^3\oplus \C^3,\qquad s\in\R.
\ee
In the Bloch transform equation  (\re{JDi}) formally reads
\be\la{Ble}
\dot{\ti Y}(\theta,t)=\ti A(\theta)\ti Y(\theta,t)\quad\mbox{\rm for \,\,a.e.}\,\,\theta\in\Pi^*,
\qquad t\in\R,
\ee
where $\ti Y(\cdot,t)\in \cX^0(\T)$
(see \ci[(8.6)]{KKpl2015}).
The Hamilton representation (\re{AJB}) implies that
\be\la{tiAJB}
 \ti A(\theta)=J\ti B(\theta),\qquad \theta\in\Pi^*\setminus \Ga^*,
 \ee
  where the Bloch energy operators $\ti B(\theta)$
are selfadjoint in  $\cX^0(\T)$.
The main crux here is that
the generator  $\ti A(\theta)$ is not selfadjoint 
and even is not symmetric in the Hilbert space $\cX^0(\T)$.
Hence we cannot diagonalize it
using the von Neumann spectral theorem.
Thus, even an introduction of the `dispersion relations' $\om_k(\theta)$, which are the eigenvalues of $\ti A(\theta)$, 
is a nontrivial problem in our situation.
Let us denote
\be\la{PI+}
\Pi^*_+:=\{ \theta\in \Pi^*\setminus \Ga^*:\Si(\theta)>0 \}. 
\ee
This is an open set of
the complete Lebesgue measure, i.e.,  $\mes(\Pi^*\setminus\Pi^*_+ )=0$, by (\re{W1}).
The key role in our approach is played by the positivity
\be\la{Hpos2}
\langle \ti Y, \ti B(\theta)\ti Y\rangle\ge \vka(\theta)\Vert \ti Y\Vert_{{\cX^1}(\T)}^2,
\quad\ti Y\in{\cX^1}(\T),\quad \theta\in\Pi^*_+,
\ee
with $\vka(\theta)>0$,  the brackets  denoting the scalar product in ${\cX^0}(\T)$.
We
have proved this bound in \ci{KKpl2015}
under   conditions  (\re{W1})--(\re{Wai})
on the ion charge density $\si$.

In present paper we
use this positivity
to show
that the eigenvectors of $\ti A(\theta)$ span the Hilbert space
$\cX^0(\T)$ by
our spectral theory of the Hamilton operators with positive energy \ci{KK2014a,KK2014b}.
This is a
special version of the  Gohberg--Krein--Langer theory of
selfadjoint operators in the Hilbert spaces with indefinite metric
 \ci[Ch.~VI]{GK} and \ci{KL1963,L1981}.
 Namely, setting $\ti\Lam(\theta):=\ti B^{1/2}(\theta)$, we obtain that
 \be\la{AK2}
 \ti A(\theta)=-i\ti \Lam^{-1}(\theta) \ti K(\theta)\ti \Lam(\theta),\qquad\theta\in\Pi^*_+,
  \ee
  where  $\ti K(\theta)=\ti\Lam(\theta) iJ\ti\Lam(\theta)$
   is a selfadjoint operator  in    ${\cX^0}(\T)$.
  Hence,
  all solutions to (\re{Ble}) admit the representation
\be\la{solexp}
 \ti Y(\theta, t)=\ti\Lam^{-1}(\theta) e^{-i \ti K(\theta)t}
\ti\Lam(\theta)\ti Y(\theta,0),\quad t\in\R,
\quad \theta\in\Pi^*_+.
\ee
We prove that the spectrum of $\ti K(\theta)$ is discrete and obtain the lower estimate for the eigenvalues
 $\om_k(\theta)$ which are 
 also the eigenvalues of  $\ti A(\theta)$
and are called 
 the dispersion relations (or the 
Floquet eigenvalues).

\medskip

Further, we represent
the solution $Y(t)$ as the inversion of the Bloch transform  (\re{solexp}).
This inversion is the series of  oscillatory  integrals
with the phase functions $\om_k(\theta)$.
Using the decay (\re{rd}) we  show that
\medskip\\
i)
$\om_k(\theta)$ are piecewise real-analytic in $\theta\in\Pi^*\setminus \Ga^*$
for every $k$;
\medskip\\
ii)
If $\om_k(\theta)\not\equiv\const$, then the set
\be\la{deg}
\{\theta\in\Pi^*\setminus \Ga^*: \na\om_k(\theta)= 0,\,\,
\det\,{\rm Hess}\,\om_k(\theta)= 0\}
\ee
has the Lebesgue measure zero;
\medskip\\
iii) In the case $M=\infty$
the limit (\re{Minf}) holds
for the constant dispersion relations $\om_k(\theta)\equiv \om^*_k$.
\medskip

These properties of the phase functions provide
the asymptotics  (\re{FY53}) and  (\re{FY5}).
Finally, we establish the
absence  of singular spectrum and
the limiting absorption principle for the selfadjoint operator $K:=i\Lam A\Lam^{-1}$.
\medskip

Note that
all our methods and results extend obviously to
 equations (\re{LPS1})--(\re{LPS3})  in the case of a~general lattice
\begin{equation}\la{gG}
\Ga=\{n_1a_1+n_2a_2+n_3a_3: (n_1,n_2,n_3)\in\Z^3\},
\end{equation}
where the generators $a_k\in\R^3$ are linearly independent. In this case
 the condition \eqref{Wai} becomes
 \begin{equation}\la{Waig}
\hat\si(\ga^*)=0,\quad \ga^*\in\Ga^*\setminus 0.
\end{equation}
Here $\Ga^*$ denotes the dual lattice
$\Ga^*=\{m_1b_1+m_2b_2+m_3b_3: (m_1,m_2,m_3)\in\Z^3\}$,  where $\langle a_k,b_j\rangle=2\pi\de_{kj}$.
The condition \eqref{Waig}
relates the properties
of the ions with the  crystal geometry.
\medskip

Let us comment on previous results in these directions.
\medskip

The  Schr\"odinger--Poisson equations for  crystal were 
introduced  in 
\ci{K2014,K2015}, where the existence  of the ground states was 
estabilshed for infinite crystals.
Recently we  have proved the linear stability of these ground states 
\ci{KKpl2015}.
In \ci{KK2017a,KK2017b} we have proved 
 the orbital stability of the ground states for  the Schr\"odinger--Poisson equations with one-particle and many-particle 
   Schr\"odinger equation in the case of
 finite crystals with periodic boundary conditions.

In the Hartree--Fock model
the crystal ground state
 was constructed for the first time by Catto, Le Bris, and  Lions  \ci{CBL2001,CBL2002}.
For the Thomas--Fermi model, see \ci{CBL1998}.

In \ci{CS2012}, Canc\'es and Stoltz established the well-posedness  for
the dynamics of
local perturbations of the  ground state density matrix
in the  {\it random phase approximation}
for the reduced  Hartree--Fock equations
with the Coulomb  pairwise interaction potential $w(x-y)=1/|x-y|$.
However, the  space-periodic nuclear potential
in the equation \ci[(3)]{CS2012}
does not depend on time that corresponds to
the fixed nuclei positions.

The nonlinear Hartree--Fock dynamics
with the Coulomb potential and
without the  random phase approximation
was not studied previously,
see the discussion in
\ci{BL2005} and in the introductions of \ci{CLL2013,CS2012}.

In \ci{CLL2013}
E. Canc\`es, S. Lahbabi, and M. Lewin considered the random
reduced HF model of crystal  when
the ions charge density and the electron density matrix are random processes,
and the action of the lattice translations on the probability space is ergodic.
They obtained suitable generalizations of the Hoffmann--Ostenhof
and Lieb--Thirring inequalities  for ergodic density matrices,
and
constructed a random potential satisfying the Poisson equation
with the corresponding stationary stochastic  charge density.
The main result 
of \ci{CLL2013}
is the  coincidence of this model with the thermodynamic limit in
the case of the short-range Yukawa interaction.

In \ci{LS2014-1}, Lewin and Sabin showed the well-posedness for the
reduced von Neumann equation, describing the Fermi gas,
with density matrices of infinite trace
and pairwise interaction potentials $w\in L^1(\R^3)$. Moreover, they
proved the asymptotic stability of translation-invariant stationary states
for 2D Fermi gas \ci{LS2014-2}.

The traditional {\it one-electron} Bethe--Bloch--Sommerfeld
mathematical model of crystals is known to be  the linear
Schr\"o\-din\-ger
equation with a
space-periodic static potential, which corresponds to the standing ions.
The corresponding spectral theory
is well developed, see \ci{RS4} and the references therein.
The scattering theory for short-range and long-range perturbations of
such `periodic operators' was  constructed in \ci{GN1,GN2}.
\medskip

The first results on the dispersion decay
$\sim t^{-1}$
were obtained by Firsova \ci{Fir1996}
for 1D  Schr\"odinger equation with space-periodic
potential  
for finite band case. The proofs
 rely on Korotyaev's results \ci{Kor1991}
on stationary
points of the dispersion relations.

The decay
$\sim t^{-\ve}$ with a small $\ve>0$
for the 1D Schr\"odinger equation
with an infinite band potential
was  established by Cuccagna
\ci{Cuc2008s}.
This decay was applied to the asymptotic  stability of standing waves
in presence of
small nonlinear perturbations \ci{Cuc2006}.

The absense of constant  dispersion relations
for the periodic
Schr\"odinger equations was  established
by Thomas \ci{Thom1973}, see also Lemma~2~(c) of  \ci{RS4}, p.~308.

Recently Prill \ci{Prill2014} proved the decay $\sim t^{-p}$ with  $p=3/2$ and $p=1/2$
(under distinct assumptions)
for
the 1D Klein-Gordon equation  with
a periodic Lam\'e potential and its short range perturbations.

The dispersion decay
for the periodic Schr\"odinger and Klein-Gordon  equations
in higher dimensions $n\ge 2$ was not obtained
previously.
\smallskip

Our paper is organized as follows.
In Section 2 we recall some formulas from~\ci{KKpl2015} for
the Bloch representation.
In Section 3 we introduce the dispersion relations and prove their properties.
In Section 4 we prove the asymptotics (\re{FY53}), (\re{FY5}), and
in Section 5 we justify the limiting absorption principle.
In Appendix A we collect some formulas from \ci{KKpl2015}
which we need in our calculations.
\bigskip

{\bf Acknowledgments} The authors are grateful to Herbert Spohn for discussions and remarks.


\setcounter{equation}{0}
\section{The Bloch representation of the dynamics}
In this section we recall some notations from
 \ci{KKpl2015}  
and establish novel exact bounds and compactness
for the inversion of the Bloch generators.

\subsection{The Bloch representation of the dynamics}

We set $\cS_+:= \cup_{\ve>0} \cS_\ve$, where $\cS_\ve$  is the space of functions
$\Psi\in\cS(\R^3)$ whose Fourier transforms  $\hat\Psi(\xi)$ vanish in the $\ve$-neighborhood of the lattice $\Ga^*$,
and let  $l_c$ be  the space of compactly supported sequences $q(n)\in \R^3$.

\begin{definition}\la{dD}
Let 
$\cD:=\{Y=(\Psi_1,\Psi_2,q,p):  \Psi_1,\Psi_2\in \cS_+,~~~ q, p\in l_c \}$.
\end{definition}

Note that the space $\cD$  is dense in $\cX^0(\R^3)$, and
$A:\cD\to \cX^0(\R^3)$ by Theorem 4.2 of \ci{KKpl2015} . Denote by $\Pi$  the primitive cell
$\Pi:=[0,1]^3$ and $\cX^0(\Pi):=L^2(\Pi)\oplus L^2(\Pi)\oplus\R^3\oplus\R^3$.
For
$  Y=(\Psi_1, \Psi_2, q, p)\in \cX^0(\R^3)$  and $n\in\Z^3$ we set 
\be\la{Yn}
Y(n)=(\Psi_1(n,\cdot),\Psi_2(n,\cdot), q(n),p(n)),\quad\mbox{\rm where}\quad
\Psi_j(n,y)=
\Psi_j(n+y)\,\,\,\,\mbox{\rm for a.e.}\,\, y\in \Pi.
\ee
Obviously,
\be\la{Yn2}
\Vert Y\Vert_{\cX^0(\R^3)}^2=\sum_{n\in\Z}\Vert Y(n)\Vert_{\cX^0(\Pi)}^2.
\ee
\bd\la{dcr}
We will call the sequence $Y(n)$ as the cell representation of $Y\in \cX^0(\R^3)$.
\ed
The ground state (\re{gri}) is invariant with respect to translations of the lattice $\Ga$,
and hence the generator  $A$ commutes with these translations.
Therefore, the operator
$A$ can be reduced using the discrete Fourier transform
\be\la{F}
 \hat Y(\theta)=F_{n\to \theta}Y(n):=\sum\limits_{n\in \Z^3}e^{in\theta}Y(n)
 =(\hat{\Psi}_1(\theta,\cdot),\hat{\Psi}_2(\theta,\cdot),\hat {q}(\theta),\hat{p}(\theta))
 \quad~~{\rm for \ a.e.}~~\theta\in \Pi^*,
\ee
By the Parseval-Plancherel theorem
the series converge in $L^2(\Pi^*,\cX^0(\Pi))$  for $ Y\in \cX^0(\R^3)$.

\begin{definition}\la{FGLZ}
(see  \ci{DK2005,PST,RS4}, and
 \ci{KKpl2015})
The Bloch transform of $Y\in \cX^0(\R^3)$ is defined as
\be\la{YPi}
 \ti Y(\theta)=[\cF Y](\theta):= \cM(\theta)\hat Y(\theta):
 =(\ti{\Psi}_1(\theta,\cdot),\ti{\Psi}_2(\theta,\cdot),\hat{ q}(\theta),\hat{ p}(\theta))
 \qquad~~{\rm for \ a.e.}~~\theta\in\Pi^*,
\ee
where $\ti{\Psi}_j(\theta,y)=M(\theta)\hat{\Psi}_j:=e^{i\theta y}\hat{\Psi}_j(\theta,y)$
are  $\Ga$-periodic functions in $y\in\R^3$.
\end{definition}
The transform
$\cF:\cX^0(\R^3)\to  L^2(\Pi^*,\cX^0(\T))$ is an isomorphism
by the Parseval-Plancherel identity.
The inversion is given by
\be\la{FZI}
Y(n)=|\Pi^*|^{-1}\int_{\Pi^*}e^{-in\theta}\cM(-\theta)\ti Y(\theta)d\theta,\qquad n\in\Z^3.
\ee
 In the Bloch transform one has
$
\widetilde{AY}(\theta)=
\ti A(\theta) \ti Y(\theta)$ for   $Y\in\cD$ and  $\theta\in  \Pi^*\setminus\Ga^*$.
Here $\ti A(\theta)$ denotes the operator matrix
\be\la{tiA}
\ti A(\theta)=\!\!\left(\!\!\!
 \ba{cccl}
0 &\!\!\ti H^0(\theta) & 0  & 0\medskip\\
 -\ti H^0(\theta)-2e^2\psi^0\ti G(\theta)\psi^0&
0 &~\ti S(\theta) & 0\\
 0  & 0&\!0&\!M^{-1}\\
 -2\ti S^{\5*}(\theta)\!\!&  0 &-\ti T(\theta)&\!0\\
\ea\!\!\!\!\right),\qquad\theta\in\Pi^*\setminus\Ga^*,
\ee
where the operator entries are given by \eqref{tiHS}--\eqref{tiH1}.
The operator  $\ti A(\theta)$ admits the representation
\be\la{hess2i}
\ti A(\theta)=J\ti B(\theta),\,\,\, \ti B(\theta) =\!\left(\!\ba{cccl}
 2\ti H^0(\theta)+4e^2\psi^0 \ti G(\theta)\psi^0&  0 & 2\ti S(\theta)&0
 \medskip\\
 0 &2\ti H^0(\theta)&  0 &0
\medskip\\
 2\ti S^{\5*}(\theta) &   0
 &   \ti T(\theta)             & 0  \\
      0    &    0            &   0 &  M^{-1} \\
\ea\!\right)\!\!,\,\,\,\,\,\,\,
\theta\in\Pi^*
\setminus\Ga^*,
\ee
\begin{lemma}\la{lB}
Let conditions \eqref{L123i} and    \eqref{W1},  \eqref{Wai} hold. Then
\medskip\\
i) For $\theta\in\Pi^*\setminus\Ga^*$ the operator  $\ti B(\theta)$ is selfadjoint
in $\cX^0(\T)$
with the domain ${\cX^2}(\T)$;  the quadratic  form
$\langle\ti B(\theta) \ti Y, \ti Y\rangle$ extends by continuity to $\ti Y\in \cX^1(\T)$:
\be\la{qfB}
\langle\ti B(\theta) \ti Y, \ti Y\rangle\le C\Vert \ti Y\Vert_{\cX^1(\T)}^2,\qquad \ti Y\in \cX^1(\T).
\ee
ii) For $\theta\in\Pi^*_+$
\be\la{qfB2}
\Vert \ti Y\Vert_{\cX^1(\T)}^2\le \fr1{\vka(\theta)}\langle\ti B(\theta) \ti Y, \ti Y\rangle,
\qquad \ti Y\in \cX^1(\T).
\ee
\end{lemma}
\begin{proof}
i) The representation
$\ti A(\theta)=J\ti B(\theta)$
 follows from  (\re{AJB}).
The operator $\ti B(\theta) $
is symmetric on the domain  $\cX^2(\T)$.
Moreover,  all operators in (\re{hess2i}), except for $\ti H^0(\theta)$, are bounded by (\re{L123i}).
Finally,  $\ti H^0(\theta)$
is  selfadjoint in $L^2(\T)$ with the domain $H^2(\T)$.
Hence, $\ti B(\theta)$ is selfadjoint on  the domain $\cX^2(\T)$.
\medskip\\
ii) The bound (\re{qfB2}) holds by (\re{Hpos2}).
\end{proof}
\bc\la{lLam}
Under conditions
 \eqref{L123i} and    \eqref{W1},  \eqref{Wai}
\medskip\\
i) The operator $\ti\Lam(\theta):=\ti B^{1/2}(\theta):\cX^1(\T)\to\cX^0(\T) $ is bounded.
\medskip\\
ii)
For $\theta\in\Pi^*_+$
the operator $\ti\Lam(\theta)$ is invertible in $\cX^0(\T)$. Moreover,
\be\la{qfB22}
\Vert
\ti \Lam^{-1}(\theta) \ti Z \Vert_{\cX^1(\T)}^2
\le \fr1{\vka(\theta)}\Vert \ti Z\Vert_{\cX^0(\T)}^2,\qquad \ti Z\in \cX^0(\T)
\ee
\ec
\begin{proof}
i) $\ti\Lam(\theta)$ is bounded 
by (\re{qfB}), since
\be\la{sin}
\langle\ti\Lam(\theta) \ti Y,  \ti\Lam(\theta)  \ti Y\rangle=
\langle\ti B(\theta) \ti Y, \ti Y\rangle,\qquad \ti Y\in \cX^2(\T).
\ee
ii) $\ti\Lam(\theta)$ is invertible by the positivity  (\re{Hpos2}),
and
(\re{qfB22}) follows
by (\re{qfB2}) as applied to $\ti Y=\ti \Lam^{-1}(\theta) \ti Z$.
\end{proof}

\subsection{Reduction to selfadjoint generator}\la{sred}


\bd\la{dws}
A function $Y(t)\in C(\R,\cX^1(\R^3))$
is
a weak solution to (\re{JDi})  if, for every $V\in\cD$,
\be\la{ws}
\langle Y(t)-Y(0),V\rangle= \int_0^t\langle Y(s),A^*V \rangle ds,\qquad t\in\R.
\ee
\ed
In the Bloch transform the weak solution satisfies (\re{Ble})
in the sense of $\cX^1(\T)$-valued distributions, see \ci{KKpl2015}.
Applying $\ti\Lam(\theta)$ to both sides of
 \eqref{Ble}, we obtain the equivalent equation
 $ \dot {\ti Z}(\theta,t)=-i\ti K(\theta)\ti Z(\theta,t)$ for $\theta\in \Pi^*_+$
in the sense of vector-valued distributions,
  where ${\ti Z}(\theta,t):={\ti \Lam}(\theta){\ti Y}(\theta,t)$ and
  $\ti K(\theta):=\ti\Lam(\theta) iJ\ti\Lam(\theta)$ is formally symmetric operator in $\cX^0(\T)$.
The main  crux  is that the domain of $\ti K(\theta)$ is unknown, since the ion density
$\si(x)$ is not smooth in general, and so the PDO machinery does not apply.
The following lemma plays a~key role in our approach.

\begin{lemma}\la{lH0}
(cf. Lemma 8.2 of \ci{KKpl2015})
Let conditions
 \eqref{L123i} and    \eqref{W1},  \eqref{Wai} hold. Then
for $\theta\in\Pi^*_+$
\medskip\\
i) $\ti K(\theta)$ is a selfadjoint operator in ${\cX^0}(\T)$ with dense domain  $D(\ti K(\theta))\subset {\cX^1}(\T)$;
\medskip\\
ii) $\ti K^{-1}(\theta)$ is a
compact selfadjoint operator in  ${\cX^0}(\T)$, and
\be\la{qfB23}
\Vert
\ti K^{-1}(\theta) \ti Z \Vert_{\cX^1(\T)}^2
\le \fr C{\vka(\theta)}\Vert \ti Z\Vert_{\cX^0(\T)}^2,\qquad \ti Z\in \cX^0(\T).
\ee

\end{lemma}
\begin{proof}
{\it i)} The operator $\ti \Lam(\theta)$ is injective. Moreover, $\Ran\5\ti\Lambda(\theta)={\cX^0}(\T)$.
Hence, $\Ran\5\ti  K(\theta)={\cX^0}(\T)$. Consider the inverse operator
\begin{equation}\la{G}
 \ti R(\theta):=\ti K^{-1}(\theta)=i\ti\Lam^{-1}(\theta) J^{-1}\ti\Lam^{-1}(\theta).
\end{equation}
This operator is selfadjoint, since it is bounded and symmetric. Hence,
$D(\ti R(\theta))=\Ran\5 \ti K(\theta)={\cX^0}(\T)$.
Therefore, $\ti K(\theta)=\ti R^{-1}(\theta)$ is a densely defined selfadjoint operator
by Theorem 13.11 (b) of \cite{Rudin}:
$$
\ti K^*(\theta)=\ti K(\theta)~, \quad D(\ti K(\theta))=\Ran\5 \ti R(\theta)\subset \
\Ran \5\ti\Lam^{-1}(\theta)\subset {\cX^1}(\T),
$$
where the last inclusion follows from (\re{qfB22}).
\medskip\\
{\it ii) }  Estimate (\re{qfB23}) follows from (\re{qfB22}) and (\re{G}).
Hence, $\ti K^{-1}(\theta)$ is a
compact operator in  ${\cX^0}(\T)$ by  the Sobolev embedding theorem.
\end{proof}

As a consequence,
\be\la{ZHbi}
 \ti Z(\theta,t)=e^{-i \ti K(\theta) t}\ti Z(\theta,0),~~~~~~~
 \ti Z(\theta,0):=\ti\Lam(\theta)\ti Y(\theta,0).
\ee

The definition (\re{cW}) implies that the operator $\Lam:=\cF^{-1}\ti\Lam(\theta)$ is the isomorphism $\cW\to\cX^0(\R^3)$.

\bp\la{pd} (Corollary 8.5 of \ci{KKpl2015})
Let the positivity (\re{Hpos2}) hold. Then,
for every initial state $Y(0)\in \cW$, there exists a unique weak
 solution $Y(\cdot)\in C_b(\R,\cW)$ to equation \eqref{JDi}.
 The solution is given by formula \eqref{solexp}.

\ep




\setcounter{equation}{0}
\section{Dispersion relations}\la{dr}
  Here we establich the properties of the eigenvalues of  $\ti K(\theta)$ which play the key role in the proof of the 
   dispersion decay.
Lemma \re{lH0} implies the spectral resolution
\be\la{Hsr}
 \ti K(\theta)=\sum_{k=1}^\infty \om_k(\theta)P_k(\theta),~~~~~~\theta\in \Pi^*_+,
\ee
where  $\om_k(\theta)$
 are the eigenvalues
(dispersion relations)
counted with their multiplicities,
$$
|\om_1(\theta)|\le |\om_2(\theta)|\le\dots,
$$
and
$P_k(\theta)$ are the corresponding orthogonal projections.

\bl \la{lQ}
Let   conditions
 \eqref{L123i} and    \eqref{W1},  \eqref{Wai} hold and
$Q$ be a compact subset of $\Pi^*_+$. Then
\be\la{omk1}
|\om_k(\theta)|\ge \ve(Q)k^{2/3},\qquad k\ge 1,\qquad\theta\in Q,
\ee
where $\ve(Q)>0$.
\el
\begin{proof}
The key role in the proof of (\re{omk1}) is played by the estimate \ci[(7.23)]{KKpl2015}:
\be\la{com}
b(Q):=
\inf_{\theta\in Q}\vka(\theta)>0
\ee
for any compact subset $Q\subset\Pi^*_+$.  The  expansion (\re{Hsr})  implies that 
\be\la{Hsr2}
 |\ti K^{-1}(\theta)|=\sum_{k=1}^\infty |\om_k(\theta)|^{-1}P_k(\theta),~~~~~~\theta\in \Pi^*_+.
\ee
Moreover,  by duality we have from estimate (\re{qfB23})  
\be\la{qfB24}
\Vert
\ti K^{-1}(\theta) \ti Z \Vert_{\cX^0(\T)}^2\le \fr C{b(Q)}\Vert \ti Z\Vert_{\cX^{-1}(\T)}^2,\qquad \ti Z\in \cX^0(\T),\quad \theta\in Q
\ee
due to (\re{com}), since the operator $\ti K^{-1}(\theta)$ is selfadjoint.
At last, the norm in the right-hand side of (\ref{qfB24}) can be written
as $\Vert   g \ti Z\Vert_{\cX^{0}(\T)}$, where
\be\la{g}
g=\left(\ba{cccc}
(-\De+1)^{-1/2}&0&0&0\\
0&(-\De+1)^{-1/2}&0&0\\
0&0&1&0\\
0&0&0&1
\ea
\right)
\ee
is the positive selfadjoint operator in $\cX^{0}(\T)$. Now (\re{qfB24}) gives that
\be\la{qfB25}
\Vert
|\ti K^{-1}(\theta) | \ti Z \Vert_{\cX^0(\T)}\le C(Q)\Vert g \ti Z\Vert_{\cX^{0}(\T)},\qquad \ti Z\in \cX^0(\T),\quad \theta\in Q.
\ee
Hence,
the Rayleigh-Courant-Fisher theorem  (\ci[Theorem 1, p.110]{A}
and \ci[Theorem XIII.1, p.91]{RS4}) implies that
\be\la{RCF}
|\om_k(\theta)|^{-1}\le C(Q) g_k,\qquad k\ge 1,\quad  \theta\in Q,
\ee
 where $g_1\ge g_2\ge\dots$ are the eigenvalues of $g$ counted with their multiplicities.
 Therefore, (\re{omk1}) holds, since $g_k\le Ck^{-2/3}$.
 The last inequality is obvious, in as much as $k\le \#(n\in\Z^3:n^2+1 \le g_k^{-1})\le C_1 g_k^{-3/2}$.
  \end{proof}
Further we use the exponential decay of the ion charge density (\re{rd}).
It is easy to construct examples of densities $\si$ satisfying
 all conditions of Theorem \re{tmg}:
(\re{L123i}), (\re{rd})  and the Wiener and Jellium conditions
(\re{W1}), (\re{Wai}).
\begin{example}\la{ex2}
{\rm
 For example, all these conditions hold for
$\si(x_1,x_2,x_3):=\si_1(x_1)\si_1(x_2)\si_1(x_3)$,
where
$$
\ti\si_1(\xi):=\fr{\sin\fr\xi2}\xi e^{-\xi^2},\qquad\xi\in\R.
$$
 In particular,  (\re{rd})  holds by the Paley--Wiener theorem.
}
\end{example}

 The condition (\re{rd}) implies that
 the function  $\ti\si(\theta,y)$ is analytic with respect to~$\theta$ in the complex tube
\[
 \Pi^*_\ve:=\{\theta\in [\Pi^*\setminus\Ga^*]\oplus i\R^3:~|\rIm\theta|<\ve\}.
\]
Hence, the finite rank operators $\ti S(\theta)$
and $\ti T(\theta)$ defined in (\ref{tiHS}) -- (\ref{tiH1})
are also analytic in $\theta\in \Pi^*_\ve$. Therefore, $\ti K(\theta)$
is real-analytic on  $\Pi^*_+$.
Denote the set
\be\la{spec}
\cR:=\{(\theta,\om): \,\,\theta\in  \Pi^*_+, \,\,\om\in\spec\ti K(\theta)\}.
\ee
The eigenvalues $\om_k(\theta)$ and  the projections $P_k(\theta)$ become 
 single-valued functions on $\cR$: for $R=(\theta,\om_k(\theta))$
 \be\la{func}
 \theta(R):=\theta,\qquad
 \om(R):=\om_k(\theta),\qquad
  P(R):=P_k(\theta).
 \ee
 These functions are
  continuous on the manifold
  $\cR$ endowed with natural topology  by the incluzion $\cR\subset \Pi^*\times\R$.
 They are
  piecewise analytic on $\cR$
 by the following lemma,
 which extends \ci[Lemma 1.1]{S1989} from the Schr\"odinger equation with 
 periodic potential to the system (\re{JDi}).
\begin{lemma}\la{lom}
Let  all conditions of Theorem \re{tmg} hold.
Then
for every point $R^*=(\theta^*,\om^*)\in\cR$  there exists a neighborhood
$U=U(R^*)\subset \cR$ with its projection $V=V(R^*)$ onto $\Pi^*_+$\, ,
and  a critical subset
$\cC=\cC(R^*)\subset \Pi^*_\ve$, which is a finite union of analytic   submanifolds
of positive complex codimension in $\Pi^*_\ve$,
with the following properties:
\medskip\\
i) For any point
$R=(\theta,\om) \in U$
we have $\om(R):=\om\in \spec \ti K(\theta)$.  
\medskip\\
ii) For  any point $\theta' \in V\setminus \cC$  
there exists a neihborhood $W=W(\theta')\subset V\setminus \cC$ such that
$R=(\theta,\om) \in U$ with $\theta\in  W$ if and only if 
$\om=\om_l(\theta)$
with some $l=1,..., L=L(R^*)$. 
\medskip\\
iii)  The eigenvalues $\om_l(\cdot)$ and the corresponding projections  $P_l(\cdot)$ are  real-analytic functions  on   $W$
and admit an analytic continuation outside $\cC$ in a complex neighborhood of $\theta^*$ in $\Pi^*_\ve$.
\medskip\\
iv) For each $l=1,..., L(R^*)$, either
\be\la{Ck}
  \na \om_{l}(\theta)\ne 0,~~~~\theta\in W,
\ee
or
\be\la{omc}
  \om_{l}(\theta)\equiv \om^*,~~~~~~\theta\in W.
\ee
v) If  (\re{omc}) holds with some $l$ for a point $R^*=(\theta^*,\om^*)$, 
then 
the constant eigenvalue
also exists for
 $(\theta,\om^*)$ with any
$\theta\in\Pi^*_+$.
\end{lemma}
\begin{proof}
Let us set $r:=\dist (\om^*, \spec \ti K(\theta^*)\setminus\om^*)>0$. Then
\be\la{Rp1}
 P(\theta)=-\fr1{2\pi i}
 \int_{|\om-\om^*|=r/2} [\ti K(\theta)-\om]^{-1}d\om
\ee
is a finite-rank 
Riesz 
 projection, which is analytic in a complex neighborhood of $\theta^*$.
Its range $\Ran P(\theta)$ is invariant under $\ti K(\theta)$, and hence the bifurcated  from
$\om^*$ eigenvalues of $\ti K(\theta)$ coincide with the roots of the characteristic equation
\be\la{Rp2}
 \det[M(\theta)-\om]=0,
\ee
where
$M(\theta):=\ti K(\theta)|_{\Ran P(\theta)}$.
The coefficients of this polynomial are analytic
functions of $\theta$
in a complex neighborhood of $\theta^*$, and hence
i)--iv) follow
by the arguments from the proof of  Lemma 1.1 of \ci{S1989}.

Finally,  v) follows from the fact that the set of the corresponding $\theta\in\Pi^*_+$ is closed and open at the same time by the analyticity
of each $\om_l(\theta)$ in a connected open region of  $\Pi^*_\ve\setminus\cC$.
\end{proof}

\begin{definition}\la{dOm*}
 $\Om^*$ is the set of all 
 $\om^*$ which are 
constant eigenvalues (\re{omc}) at least for one point
$R^*\in\cR$. 
\end{definition}


\setcounter{equation}{0}
\section{Dispersion decay}
Here we prove our main Theorem \re{tmg}.
Recall that
$\Lam: \cW\to\cX^1(\R^3)$  is an isomorphism by  the definition (\re{cW}), and hence,
 it suffices to check the corresponding asymptotics for $Z(t):=\Lam Y(t)\in C(\R,\cX^0(\R^3))$:
\be\la{FY52}
 Z(t)=\sum_1^M Z_k e^{-i\om^*_k t}+Z^c(t);
\qquad \Vert Z^c(t)\Vert_\al \to 0~,\quad |t|\to\infty,
\ee
where $Z_k\in\cX^0(\R^3)$ and $ \al<-3/2$.
Substituting (\re{ZHbi}) for $\ti Z(\theta,t)$
into the inversion formula  (\re{FZI})
we obtain the corresponding cell representation
\be\la{FY2}
 Z(n,t)=|\Pi^*|^{-1}\int_{\Pi^*} e^{-in\theta}\cM(-\theta)e^{-i\ti K(\theta)t}
 \ti Z(\theta,0)d\theta,\quad n\in\Z^3.
\ee
The weighted norms  (\re{cXw}) are equivalent to the modified norms 
\be\la{eqn}
|\!|\!| Z|\!|\!|_\al ^2:=\sum_{n\in\Z^3}(1+|n|)^{2\al} \Vert Z(n) \Vert_{\cX^0(\Pi)}^2,\qquad  Z\in\cX^0(\R^3),
\ee
where $Z(n)$
 are defined by (\re{Yn}). Hence, 
 the decay  (\re{FY52}) for $Z^c(t)$ is equivalent to
 \be\la{dZc}
 \sum_{n\in \Z^3}(1+|n|)^{2\al} \Vert Z^c(n,t) \Vert_{\cX^0(\Pi)}^2\to 0,
 \quad t\to\infty.
\ee
The spectral resolution
 (\re{Hsr}) implies that
\be\la{FY22}
 Z(n,t)=|\Pi^*|^{-1}\int_{\Pi^*}e^{-in\theta}\cM(-\theta)
 [\sum_k e^{-i\om_k(\theta)t}P_k(\theta)]\ti Z(\theta,0)d\theta,\qquad n\in\Z^3.
\ee
Equivalently,
\be\la{FY222}
 Z(n,t)=|\Pi^*|^{-1}\int_{\cR}e^{-in\theta}\cM(-\theta)
 e^{-i\om t} P(\theta,\omega)\ti Z(\theta,0)  d\theta,\qquad n\in\Z^3, 
\ee
where  $\theta$, $\om$ and the projection $P(\theta,\omega)$ are the single-valued continuous functions
(\re{func}) on $\cR$. We denote by $d\theta$ is the corresponding differential form on $\cR$.
The integral is well defined by Lemma \re{lom}. 
\medskip

\subsection{Discrete spectral component}  
We define the series of oscillating terms of (\re{FY52}) by its cell representation
\be\la{discr}
\sum_k Z_k (n)e^{-i\om^*_k t}=
|\Pi^*|^{-1}\int_{\{(\theta,\om)\in\cR:\om\in \Om^*\}}e^{-in\theta}\cM(-\theta)
 e^{-i\om t} P(\theta,\omega)\ti Z(\theta,0)  d\theta,\qquad n\in\Z^3.
\ee
Now  (\re{Minf}) follows from (\re{omk1}).
\medskip

\subsection{Continuous spectral component} 
It remains to prove the decay (\re{FY52}) for the remainder corresponding to the cell representation
\be\la{FY222d}
 Z^c(n,t)=|\Pi^*|^{-1}\int_{\cV}e^{-in\theta}\cM(-\theta)
 e^{-i\om t} P(\theta,\omega)\ti Z(\theta,0)  d\theta,\qquad n\in\Z^3,
\ee
where the integration  spreads over the set
$\cV:=\{(\theta,\om)\in\cR:\om\not\in \Om^*\}$.
For every $\nu>0$, we split $Z^c(t)=Z^\nu_-(t)+Z^\nu_+(t)$, where
\beqn
 Z^\nu_-(n,t)&=&|\Pi^*|^{-1}\ds\int_{\cV^\nu_-} e^{-in\theta}\cM(-\theta)
  e^{-i\om t}P(\theta,\om)\ti Z(\theta,0)d\theta,
 \la{Znu}\\
 \nonumber\\
 \la{Rnu}
  Z^\nu_+(n,t)&=&|\Pi^*|^{-1}\ds\int_{\cV^\nu_+}e^{-in\theta}\cM(-\theta)
  e^{-i\om t}P(\theta,\om)\ti Z(\theta,0)d\theta.
\eeqn
Here $\cV^\nu_-:=\{(\theta,\om)\in\cV:  |\om|\le\nu\}$ and
$\cV^\nu_+:=\{(\theta,\om)\in\cV:  |\om| >\nu\}$.
\medskip\\
  {\bf  High energy  component.}
By  (\re{Yn2}) and the
Parseval--Plancherel theorem
\be\la{Zn}
\Vert  Z^\nu_+(t)\Vert_{\cX^0(\R^3)}^2=\sum_{n\in\Z^3} \Vert  Z^\nu_+(n,t)\Vert_{\cX^0(\Pi)}^2=
 |\Pi^*|^{-1}\int_{\cV^\nu_+}
\Vert P(\theta,\om)\ti Z(\theta,0)\Vert_{\cX^0(\T)}^2 d\theta.
 \ee
 According to definition (\re{cW})
the condition $Y(0)\in\cW$ means that $Z=\Lam Y(0)\in\cX^0(\R^3)$.
Hence, the Parseval--Plancherel identity gives
\be\la{Wm}
\Vert Z(0)\Vert^2_{\cX^0(\R^3)}=|\Pi^*|^{-1} \int_{\Pi^*} \Vert\ti Z(\theta,0)\Vert_{\cX^0(\T)}^2 d\theta<\infty.
\ee
Therefore,    (\re{Zn}) implies that
\be\la{Zn2}
\Vert  Z^\nu_+(t)\Vert_{\cX^0(\R^3)}
\to 0, \qquad \nu\to\infty
\ee
uniformly in $t\in\R$
by the $\si$-additivity
since $\cap_{\nu>0} \cV_+^\nu=\emptyset $. 
\medskip\\
{\bf Low energy component.}
It remains  to prove the decay (\re{FY52}) for $Z^\nu_-(t)$ corresponding to 
the cell representation
$Z^\nu_-(n,t)$.
It suffices to check that every norm
$\Vert Z^\nu_-(n,t) \Vert_{\cX^0(\Pi)}$ decays to zero as $t\to\infty$, since $\al<-3/2$ and
\be\la{dZc2}
 \sum_{n\in \Z^3} \Vert Z^\nu_-(n,t) \Vert_{\cX^0(\Pi)}^2=\Vert  Z^\nu_-(t) \Vert_{\cX^0}^2=\const, \qquad t\in\R
 \ee
 by (\re{Yn2}) and formula of type (\re{ZHbi}) for $\ti Z^\nu_-(\theta,t)$.
 \medskip\\
{\bf Reduction to a compact set and partition of unity} 
Consider an open precompact subset $Q\subset\Pi^*_+$ such that the Lebesgue measure of
$\Pi^*_+\setminus Q$ is sufficiently small, and denote
$\hat Q^\nu:=\{R=(\theta,\om)\in\cV^\nu_-: \theta\in Q\}$.
Then
the $\cX^0(\Pi)$-norm of the integral  of type (\re{Znu})  over
$\cV^\nu_-\setminus \hat Q^\nu$ is small uniformly in $t\in\R$ by (\re{Wm}).
Hence,
it remains to prove the decay for
\be\la{Zn3}
Z^\nu_Q(n,t):=|\Pi^*|^{-1}\int_{\hat Q^\nu}
e^{-in\theta}\cM(-\theta)
  e^{-i\om t}P(\theta,\omega)\ti Z(\theta,0)d\theta.
\ee
The  asymptotics (\re{omk1}), which are
uniform in $\theta\in Q$, imply that the set $\hat Q^\nu$
is open and precompact in $\cR$. 
Neglecting an arbitrarily  small term we can assume that $Q$ does not intersect a small neighborhood of the
critical submanifold
 $\cC_j\subset V(\theta_j)$ for every $j$.
Hence, we can cover $\hat Q^\nu$
 by a finite number of neighborhoods
$W(R_j)$ from Lemma \re{lom} 
with $R_j=(\theta_j,\om_j)\in\ov{\hat Q^\nu}$.
Then there exists a 
partition of unity
$\chi_j\in C(\cR)$   with $\supp\chi_j\subset W(R_j)$:
\be\la{pu}
 \sum_j\chi_j(R)=1,\quad R=(\theta,\omega)\in \hat Q^\nu.
 \ee
 Hence, (\re{Zn3}) becomes the finite sum
\be\la{FY223}
 Z^\nu_{jl}(n,t)=
 \sum_{j,l}
 |\Pi^*|^{-1}
 \int_{W(R_j)}
 e^{-in\theta}
 \chi_j(\theta,\om_{jl}(\theta))\cM(-\theta)e^{-i\om_{jl}(\theta)t}P_{jl}(\theta)\ti Z(\theta,0)d\theta,
\ee
where the functions $\om_{jl}$ and projections $P_{jl}$ are constructed in Lemma \re{lom}. 
Note, that  all constant dispersion relations (\re{omc}) are excluded from the integration (\re{FY223}), and hence,
the remaining nonconstant dispersion relations $\om_{jl}(\theta)$ satisfy (\re{Ck}).
 Let us approximate
\medskip\\
i) $\chi_j(\theta,\om_{jl}(\theta))$ by $\chi_{jl}(\cdot)\in C_0^\infty(W(R_j))$ and
\medskip\\
ii) $P_{jl}(\theta)\ti Z(\theta,0)$
by some  functions $D_{jl}(\theta)\in C^\infty(W(R_j),\cX^0(\T))$  in the norm of $L^2(Q,\cX^0(\T))$.
\medskip\\
Then the corresponding error in (\re{FY223})  is small
in the norm $\cX^0(\Pi)$ uniformly in $n\in\Z^3$ and $t\in\R$.
Finally,  (\re{Ck}) implies by a partial integration the decay of the integrals  (\re{FY223})
with $\mu_{jl}(\theta)D_{jl}(\theta)$ instead of $\chi_j(\theta,\om_{jl}(\theta))P_{jl}(\theta)\ti Z(\theta,0)$ .
\bo

\setcounter{equation}{0}
\section{Spectral properties of the selfadjoint generator}
Here we study
spectral properties of the 
operator $K:=\cF^{-1}\ti K\cF$, where $\ti K$ denotes the operator of multiplication by $\ti K(\theta)$ 
in the Hilbert space
 $L^2(\Pi^*,\cX^0(\T))$. 
  \bl
$K$  is  a selfadjoint operator in  $\cX^0(\R^3)$ with a dense domain $D(K)$. 
 \el
 \begin{proof}
 Lemma \re{lH0} ii) implies that the operator $\ti K^{-1}$ of multiplication by $\ti K^{-1}(\theta)$ is  selfadjoint and injective in $L^2(\Pi^*,\cX^0(\T))$.
 Hence,  its inverse  $\ti K(\theta)$ is densely defined  selfadjoint operator  in $L^2(\Pi^*,\cX^0(\T))$ by Theorem 13.11 (b) of \ci{Rudin}.
 \end{proof}
\bc
The Hamilton generator $A$ from (\re{JDi})
admits the representation
\be\la{AK2f}
 AY=-i\Lam^{-1}K \Lam Y,\qquad Y\in\Lam^{-1} D(K),
  \ee
where $\Lam:=\cF^{-1}\ti\Lam(\theta):\cW\to\cX^0(\R^3)$ is the isomorphism.
\ec

By (\re{Hsr}),
\be\la{FY24}
 KZ(n)=|\Pi^*|^{-1}\int_{\Pi^*_+}e^{-in\theta}\cM(-\theta)
 \sum_k \om_k(\theta) P_k(\theta)\ti Z(\theta)d\theta,~~~~~~~n\in\Z^3
\ee
for any $Z\in \cX^0(\R^3)$.
Similarly
\be\la{FY23}
 Z(n)=|\Pi^*|^{-1}\int_{\Pi^*_+}e^{-in\theta}\cM(-\theta)
 \sum_k P_k(\theta)\ti Z(\theta)d\theta,~~~~~~~n\in\Z^3.
\ee
Therefore,
\be\la{FY25}
(K-\om)Z(n)=|\Pi^*|^{-1}\int_{\Pi^*_+}e^{-in\theta}\cM(-\theta)
\sum_k (\om_k(\theta)-\om) P_k(\theta)\ti Z(\theta)d\theta.
\ee
Hence, the discrete spectrum $\si_p(K)$ consists of constant dispersion relations.

\begin{lemma}\la{leig}
Let all conditions of Theorem \re{tmg} hold. Then
\medskip\\
i)
$\si_p(K)=\Om^*$.
\medskip\\
ii) The multiplicity of every  eigenvalue is  infinite.
\end{lemma}
\begin{proof}
Let $\om^*\in\Om^*$ is a constant eigenvalue  (\re{omc}) corresponding to a point $R^*=(\theta^*,\om^*)\in\cR$.
Let us take
any    $Z\in\cX^0(\R^3)$ with the Bloch transform $\ti Z(\theta) \in\Ran P(\om^*)$ for
$\theta\in V(\theta^*)$ and $\ti Z(\theta)\equiv 0$ for $\theta\not\in V(\theta^*)$.
Then 
(\re{omc}) and (\re{FY25})  imply that $(K-\om^*)Z=0$. Obvioulsy, the space of such $Z$ is infinite dimensional.
\medskip\\
Conversely, let $(K-\om^*)Z=0$ for some $Z\in\cX^0(\R^3)$, and $\ti Z(\theta^*)\ne 0$. Then (\re{FY25}) implies (\re{omc}) with some $l=1,...,L(\theta^*,\om^*)$.
\end{proof}

Let us show   that the continuous  spectrum of $K$ is absolutely continuous.
First,
(\re{FY25}) implies that
 the resolvent  $R_K(\om):=(K-\om)^{-1}$ for $\rIm\om\ne 0$ is given by
\be\la{FY26}
R_K(\om)Z(n)=|\Pi^*|^{-1}\int_{\Pi^*_+}e^{-in\theta}\cM(-\theta)
\sum_k (\om_k(\theta)-\om)^{-1} P_k(\theta)\ti Z(\theta)d\theta,\qquad Z\in\cX^0(\R^3).
\ee
Denote by $\cX_d$ the space  of discrete spectrum of $K$.
\bl\la{lss}
Let all conditions of Theorem \re{tmg} hold. Then
the singular spectrum of $K$ is empty.
\el
\begin{proof}
This follows by  Theorem XIII.20 of  \ci{RS4}. Namely, it suffices to check
the corresponding  criterion
\be\la{crit}
\sup_{0<\ve<1}
\int_a^b|\rIm\langle Z, R_K(\om+i\ve) Z\rangle |^pd\om<\infty
\ee
with any $a,b\in\R$ and some $p>1$ for a dense set of $Z\in \cX_d^\bot$.
For example, for  the linear span of vectors  $Z\in\cX^0(\R^3)$ with the Bloch transform
\be\la{YV}
\ti Z(\theta)=P_l(\theta)D(\theta),~~~~~~~~
 D\in C_0^\infty(W, \cX^0(\T)),
\ee
as constructed in Lemma \re{lom} for each $R^*=(\theta^*,\om^*)\in\cR$,
where $P_l(\theta)$ is  the projection
corresponding to an eigenvalue $\om_l(\theta)$ satisfying   (\re{Ck}).
It suffices to check (\re{crit})
only for the vectors of type (\re{YV}). Applying Sokhotski-Plemelj's formula, we obtain  for these vectors
\beqn\la{crit2}
\rIm\langle Z, R_K(\om+i\ve) Z\rangle
&=&\,\,\,\int_W\rIm\langle P_l(\theta)D(\theta), (\om-\om_l(\theta)-i\ve)^{-1} P_l(\theta)D(\theta)\rangle_{\cX(\T)}d\theta
\nonumber\\
&\to\!\!&\!\!-\!\pi\!\!\int_{\om_l(\theta)=\om}\fr{\langle P_l(\theta)D(\theta),P_l(\theta)D(\theta)\rangle_{\cX(\T)}}
{|\na\om_l(\theta)|}d\theta,\qquad \ve\to 0+,
\eeqn
which implies (\re{crit}) with any $p\ge 1$.
\end{proof}

In concluzion, let us prove the Limiting Absorption Principle. Let us denote by $\cX_\al$ the Hilbert space
of functions with the finite norm (\re{eqn}).

\begin{lemma}\la{llap}
Let all conditions of Theorem \re{tmg} hold,
and let $Z\in\cX_d^\bot$ be
a finite linear combination of
the vectors
with the Bloch transform
of type   (\re{YV}).
Then   for any
$\om\in \R$ and
$\al<-7/2$
\be\la{lap}
R_K(\om\pm i\ve)Z\toYs R_K(\om\pm i0)Z,~~~~~~\ve\to +0.
\ee
\end{lemma}
\Pr
It suffices to prove
(\re{lap}) for every vector  of type   (\re{YV}).
By (\re{FY22})
the corresponding solution $Z(t)$ with $Z(0)=Z$ reads
\[
 Z(n,t)=|\Pi^*|^{-1}\int_W e^{-in\theta}
 \cM(-\theta)e^{-i\om_l(\theta)t}P_l(\theta)\ti Z(\theta)d\theta,\quad n\in\Z^3
\]
The partial integration shows   the time-decay
\[
 \Vert Z(n,t)\Vert_{\cX^0(\Pi)}\le C(1+|n|)^2(1+|t|)^{-2}.
\]
Hence,
\[
 \Vert Z(t)\Vert_{\cX_\al}\le C(1+|t|)^{-2}.
\]
Now the convergence (\re{lap}) follows from the integral representation
\[
\qquad\qquad
\qquad
\qquad
\qquad
R_K(\om\pm i\ve)Z=\int_0^{\pm \infty} e^{(i\om\mp\ve)t}Z(t)dt.
\qquad\qquad\qquad\qquad\qquad\qquad\qquad\qquad\qquad\qquad
\Box
\]



\appendix

\protect\renewcommand{\thesection}{\Alph{section}}
\protect\renewcommand{\theequation}{\thesection.\arabic{equation}}
\protect\renewcommand{\thesubsection}{\thesection.\arabic{subsection}}
\protect\renewcommand{\thetheorem}{\Alph{section}.\arabic{theorem}}

\setcounter{equation}{0}
\section{Matrix entries of the Bloch generators}
Let us recall some notations from \ci{KKpl2015}.
For $f\in C_0^\infty(\R^3)$ the Fourier transform is defined by
\be\la{Fu2}
f(x)=\fr 1{(2\pi)^3}\int_{\R^3} e^{-i\xi x}\ti f(\xi)d\xi,\qquad x\in \R^3;
\qquad
\ti f(\xi)=\int_{\R^3} e^{i\xi x}f(x)\,dx,\qquad \xi\in \R^3.
\ee
For the real ground state (\re{gri})
the generator $A$ of the linearized dynamics
 is given by  \eqref{JDi}, where
$ S$  denotes the operator with the `matrix'
$$
 S(x,n):=e\psi^0(x)G\na\si(x-n),\qquad x\in\R^3,\,\,n\in\Z^3
$$
by formula (3.3) of \ci{KKpl2015}
and  $T$ is the real matrix with entries
$$
T(n-n'):=-\ds\langle  G\na\otimes\na\si(x-n'),  \si(x-n) \rangle
+  \ds\langle\Phi^0,\na\otimes \na\si\rangle\de_{nn'}
$$
by formula (3.4) of \ci{KKpl2015}.
The operators
$G\psi^0: L^2(\R^3)\to L^2(\R^3)$ and $S:l^2:=l^2(\Z^3)\otimes \C^3\to L^2(\R^3)$
are not bounded due to the `infrared divergence' at $\xi=0$.
On the other hand, the
operator $T$  is  bounded
in $l^2(\Z^3)\otimes \C^3$
by Lemma 3.1 of \ci{KKpl2015}.
In the Bloch representation all these operators are given by
\beqn\la{tiHS}
\ti S(\theta)=
e\psi^0\ti G(\theta)(\na-i\theta)\ti\sigma(\theta,\cdot),&& \ti G(\theta)=(i\na+\theta)^{-2},\\
\nonumber\\
\ti H^0(\theta)=\fr 12(i\na+\theta)^2-e\Phi^0-\om^0,&& \ti T(\theta)=\ti T_1(\theta)+\ti T_2+\cO(e^4)\,\,\mbox{as}\,\,e\to 0,    \la{tiH0}\\
\nonumber\\
\la{tiH1}
\ti T_1(\theta)=\sum_m\Big[  \fr{\xi\otimes\xi}{|\xi|^2}|\ti\si(\xi)|^2    \Big]_{\xi=2\pi m-\theta},&&
\ti T_2=-\sum_{m\ne 0}\Big[  \fr{\xi\otimes\xi}{|\xi|^2}|\ti\si(\xi)|^2    \Big]_{\xi=2\pi m}.
\eeqn
in accordance with the formulas (6.22)--(6.24), and (10.4), (10.12) of \ci{KKpl2015}.

\begin{remarks}\la{rb} \rm
i) $\ti T_2=0$ under the Jellium condition (\re{Wai}).
\medskip\\
ii) The operators $\ti G(\theta):L^2(\T)\to H^2(\T)$ are bounded for
$\theta\in \Pi^*\setminus\Ga^*$; however
$\Vert \ti G(\theta)\Vert \sim d^{-2}(\theta)$, where $d(\theta):=\dist(\theta,\Ga^*)$.

\end{remarks}







\end{document}